\documentclass[graybox]{svmult}

\usepackage{amsmath}
\usepackage{enumitem} 
\usepackage{xcolor}

\usepackage{mathptmx}       
\usepackage{helvet}         
\usepackage{courier}        
\usepackage{type1cm}        
\usepackage{makeidx}         
\usepackage{graphicx}        
\usepackage{multicol}        
\usepackage[bottom]{footmisc}

\makeindex             

\begin{document}

\title*{Proportional Consistency of Apportionment Methods}
\author{Michael A. Jones, David McCune, and Jennifer M. Wilson}
\institute{Michael A. Jones \at Mathematical Reviews, Ann Arbor, MI 48103 USA \email{maj@ams.org}
\and David McCune \at William Jewell College, Liberty, MO 64068 USA \email{mccuned@william.jewell.edu}
\and Jennifer M. Wilson \at The New School, New York City, NY 10011 USA \email{wilsonj@newschool.edu}}
%
%
\maketitle


\abstract{We analyze a little-known property of apportionment methods  that
captures how allocations scale with the size of the house: specifically, if, for a fixed
population distribution, the house size and allocation can be scaled down within the
set of integers, then the apportionment should be correspondingly scaled down. Balinski and Young \cite{BY} include this property  among the minimal requirements for a  “reasonable” apportionment method.
We argue that this property is better understood as a consistency requirement
since quota-based apportionments that are “less proportional” meet this requirement
while others that are “more proportional” do not. We also show that the family of
quotatone methods based on stationary divisors (including the quota method) do not satisfy this property.}

\section{Introduction}\label{intro}

Apportionment methods are typically used to  determine the number of representatives allocated to each state or district in a national parliament in proportion to their population. They are also used to determine both the number of seats awarded to political parties based on their share of the popular vote and the number of delegates awarded to presidential candidates in the US presidential primaries based on their performance in the state primaries.  The difficulty is that the apportioned values must be nonnegative integers. A thorough introduction into the methods and their properties in these contexts can be found in \cite{BY}, \cite{P}, and \cite{JMW}, respectively.

Proportional consistency addresses in an indirect way that as the total number of seats increases, the apportionment should get closer to proportionality.   This mirrors our intuitive sense that if there are more seats available then it should be easier to allocate the seats to states in a way that is closer to their fair share.  Although Balinski and Young \cite{BY} include proportional consistency as one of three properties required for an apportionment method to be proportional, the property has not been systematically studied. They state, without proof, that all reasonable methods satisfy this requirement.  Although true for all standard apportionment methods, it is not true of all apportionment methods currently in use or well-known in the literature. In particular, proportional consistency is not satisfied by  several methods used in the US presidential primaries \cite{JMW}, nor does it apply to the quotatone methods that were introduced by Balinski and Young \cite{BYQ} to satisfy quota and monotonicity conditions.

In Section \ref{prelims}, we recall several well-known apportionment methods and give a formal definition of proportional consistency and proportionality.  In Section \ref{quota}, we prove that stationary and other well-known divisor methods satisfy proportional consistency, but that not all divisor methods do. We do the same for quota-based methods.  In Section \ref{quotatone} we show that the primary family of quotatone methods do not satisfy proportional consistency.  We offer some final thoughts as a conclusion.

\section{Preliminaries}\label{prelims}

Let $H > 0$ denote the total number of representatives (i.e., the house size).  Assume state $S_i$ has population  $v_i$, $i=1, \ldots, n$, and  that  $v_1 \ge v_2 \ge \cdots \ge v_n > 0$. Let $V = \sum_i v_i$ be the total population, and let $\mathbf v = (v_1, \ldots, v_n)$. 
An apportionment method $F$ is a mapping that, for each  $(\mathbf v, H)$, assigns one or more $n$-tuples of nonnegative integers $\mathbf h= (h_1, \ldots, h_n)$ satisfying $\sum_{i \le n} h_i = H$ where $h_i$ is the number of seats allocated to state $S_i$. The range of $F$  is set-valued because of the possibility of ties at some point in the apportionment process.  Because ties occur rarely, we will write $F(\mathbf v, H) = \mathbf h$.  

Historically, the two-most important  classes of apportionment methods are  divisor and quota-based methods. A divisor method corresponds to a rounding rule $f$ which assigns to each nonnegative integer $k$, a value $f(k) \in [k,k+1]$ such that there do not exist positive integers $a$ and $b$ for which $f(a)=a-1$ and $f(b)=b$.   The value $f(k)$ acts as a bifurcation point in which positive numbers between $k$ and $f(k)$ are rounded down to $k$ and numbers between $f(k)$ and $k+1$ are rounded up to $k+1$.      

\begin{definition} The divisor method $F$ with rounding rule $f$ satisfies  $F(\mathbf v, H)=\mathbf h$ if  there exists a divisor $x > 0$ such that  $\min_{h_i>0} \frac{v_i}{f(h_i-1)} \ge x \ge \max_{h_j \ge 0}  \frac{v_j}{f(h_j)},$ where the $h_i$ are nonnegative integers that satisfy $\sum_{i \le n} h_i =H$.   \end{definition}

Divisor methods with rounding rules $f(k) = k + s$ for $s \in [0,1]$ are called {\it stationary}. These include Jefferson's (or D'Hondt's), Webster's (or Sainte-Lagu\"e's), and Adams' methods for $s=1$, $s=1/2$, and $s=0$, respectively.  Among stationary methods, Jefferson's and Webster's methods were used to apportion the US House of Representatives from 1791-1842 and from 1842-1852 and 1901-1941, respectively.  Non-stationary methods include Hill-Huntington's method---$f(k) = \sqrt{k(k+1)}$, which was used to apportion the US House from 1941-present---and Dean's method---$f(k) = \frac{k(k+1)}{k+0.5}$.   Among legislative houses using proportional representation, D'Hondt's and Sainte Lagu{\"e}'s methods are the most common. Balinski and Young \cite{BY} and Pukelsheim \cite{P}, respectively, provide the definitive account for both the mathematics and the history of the use of apportionment methods for the US House and for countries in the European Union.

Quota methods are the other well-known class of apportionment methods in which each state $i$'s quota or fair share of $H$, $q_i = H \frac{v_i}{V}$, is rounded to integers that sum to $H$.   The most well-known apportionment method is Hamilton's method, also called the greatest remainder method; it is one of a larger family of apportionment methods: shift-quota methods.  Shift-quota methods, which are based on a parameter $s \in [0, 1)$, replace the usual notion of quota with an  adjusted quota $q_i^s = (H+s)\frac{v_i}{V}$.   Let $\lfloor q_i^s \rfloor$ denote the quota rounded down to the nearest integer. Note that  $q_i^s -1 < \lfloor q_i^s \rfloor \le q_i^s$, which implies $H+s-n <\sum_i  \lfloor q_i^s \rfloor  \le H+s$, and so $H-n < \sum_i  \lfloor q_i^s \rfloor < H+1$. Integer constraints  tighten these inequalities to $H-(n-1) \le \sum_i \lfloor q_i^s  \rfloor  \le H$, ensuring that an initial allocation of  $\lfloor q_i^s \rfloor$ leaves at most $n-1$  remaining seats to distribute.  When $s=0$, the shift-quota method reduces to Hamilton's method. Other apportionments are introduced in subsequent sections.

\begin{definition}
Under the shift-quota method with parameter $s \in [0, 1)$, each state  receives an initial number of seats equal to $\lfloor  q_i^s \rfloor$. The remaining $H - \sum_i \lfloor  q^s_i \rfloor$ seats are given one each to the states with the highest fractional remainders $q_i^s- \lfloor  q_i^s \rfloor$.
 \end{definition}

\noindent {\bf Proportional Consistency and Other Properties.}  We introduce the properties of apportionment methods that are most relevant to our discussion. For additional  properties, see \cite{BY, JMW, P}.  The first property is homogeneity, which requires apportionments to be invariant under a uniform scaling of the state populations.

\begin{definition}
 A method is  {\it homogeneous} if $F(\mathbf v, H) = F(\lambda \mathbf v, H)$ for all $\lambda>0$.  
\end{definition} 

Thus if an apportionment method is homogeneous, we can refer interchangeably to $F(\mathbf v, H)$,  $F(\mathbf q, H)$ or $F(\mathbf p, H)$ where $\mathbf q$ and $\mathbf p$ are the vector of quotas $q_i$ and population proportions $p_i=v_i/V$, respectively.  The next two properties are weak proportionality and proportional consistency.  Weak proportionality requires an apportionment method to return the quotas when all quotas are integers.

\begin{definition}
$F$ is {\it weakly proportional}  if whenever $q_i$ is an integer for all $i$, then $F(\mathbf v, H) = \mathbf q$. 
\end{definition} 
 
\begin{definition} A method is   {\it proportionally consistent}  if whenever  $F(\mathbf v,  H)= \mathbf h$ and for all rational $\lambda<1$ such that $\lambda h_i$  are integers for every $i$, then $F(\mathbf v, \lambda H)= \lambda \mathbf h$.
\end{definition}

Proportional consistency means, for instance, that if an apportionment method allocates $24$ seats to 4 states by $\mathbf h = (9, 6, 6, 3)$, then it should allocate $\frac{2}{3}(24)=16$ seats by $\frac{2}{3}\mathbf h = (6, 4, 4, 2)$  and $\frac{1}{3}(24)=8$ seats by $\frac{1}{3}\mathbf h = (3, 2, 2, 1)$.  The scaling should work in only one direction: if $48$ seats were apportioned instead, we do not want to require that the apportionment to be $(18, 12, 12, 6)$. Intuitively, more seats should allow for the possibility that another apportionment could be closer proportionally.

Balinski and Young \cite{BY} define an apportionment method to be {\it proportional} if it satisfies homogeneity, weak proportionality and proportional consistency.  In Sections 3 and 4, we examine which methods satisfy proportional consistency.  The methods defined above are proportional, but we'll see other methods---including ones used in delegate apportionment---that are not proportionally consistent.

\section{Divisor and Quota Methods and Proportional Consistency}\label{quota}

Let's begin this section by proving that well-known divisor methods are proportionally consistent. Because the methods of Adams, Webster and Jefferson are stationary methods, they are included in the following proposition.

\begin{proposition}  \textcolor{white}{newline}
\vspace{-2mm}
\begin{enumerate}[label=(\roman*)]
\item Stationary divisor, the Hill-Huntington, and the Dean methods are proportionally consistent.

\item Not all divisor methods are proportionally consistent. \end{enumerate}
\end{proposition}

\begin{proof}
(i) If $F(\mathbf p, H) = \mathbf h$, then there is an $x>0$ such that 
$\min_{h_i>0} \frac{v_i}{\lambda f(h_i-1)} \ge \frac{x}{\lambda} \ge \max_{h_j \ge 0}  \frac{v_j}{\lambda f(h_j)}.$
It  suffices to show that
$\lambda f(h-1) \ge f(\lambda h-1)$ for all $h > 0$ and $\lambda  f(h) \le f(\lambda h)$ for all $h \ge 0$, 
since this implies
$$\min_{h_i>0} \frac{v_i}{f(\lambda h_i-1)} \ge  \min_{h_i>0} \frac{v_i}{\lambda f(h_i-1)} \ge   x^\prime   \ge \max_{h_j \ge 0} \frac{v_j}{\lambda f(h_j)} \ge  \max_{h_j \ge 0}  \frac{v_j}{f(\lambda h_j)};$$
hence $F(\mathbf v,  \lambda H) = \lambda  \mathbf h$. These inequalities are easily verified for rounding functions for stationary divisor methods as well as for Dean's and Hill-Huntington's  methods.

(ii) For a counterexample, let $f(k)=k+0.5$ for $k= 15, 9, 3$ and $f(k)=k+1$ for all other values of $k$.
Then if  $H = 27$, $\lambda = 2/3$, and $\mathbf q = (\frac{46}{3}, \frac{25}{3}, \frac{10}{3})$ then  $F(\mathbf q, 27) = (15,9,3)$  but $F(\mathbf q, 18) =  (11,5,2) \ne (10,6,2)$. \end{proof}

There is no formal mathematical distinction between quota-based and non-quota apportionment methods since any homogeneous apportionment method can be described in terms of quotas by a change of variables. But for convenience, we  loosely define a quota-based apportionment method as one that is easiest to define in terms of quotas, or which requires calculating the quotas as a first step and then rounding the quotas.  Let $\lceil  \cdot \rceil$ and $[ \cdot ]$ denote the ceiling and nearest integer rounding functions, respectively. The proof of the following lemma is straightforward.  

\begin{lemma}\label{quotarounding} For all $\lambda <1$, if $k_i$, $\lambda k_i$, and $t$ are positive integers, then
\vspace{-2mm}
 \begin{enumerate}[label=(\roman*)]
 \item   If $\lfloor  q_i \rfloor= k_i$ then $\lfloor \lambda q_i \rfloor= \lambda k_i$ and if $\lceil  q_i \rceil= k_i$ then $\lceil \lambda q_i \rceil= \lambda k_i.$ 

 \item If $[  q_i] =  k_i+t$ for some $t \ge 0$, then $\lambda k_i \le [\lambda q_i] \le \lambda k_i+t.$ If $[q_i] =  k_i-t$ for some $t \ge 0$, then $\lambda k_i-t \le [\lambda q_i] \le \lambda k_i.$ In both cases, if $\lambda \le 2t+1$ then $[\lambda q_i]=\lambda k_i$.
\item If  $[q_i] =  k_i$ then $[\lambda q_i]=\lambda  k_i.$
\end{enumerate}
\end{lemma}
\begin{proof} For (i), if $\lfloor  q_i \rfloor= k_i$  then $k_i  \le   q_i<  k_i+1$  or $ \lambda k_i  \le \lambda  q_i < \lambda k_i+\lambda < k_i+1$, which implies $\lfloor \lambda q_i \rfloor =\lambda k_i$. The proof of the second part of (i) is similar. 

For (ii), if $[ q_i ] =  k_i+t$ then $ k_i+t-\frac{1}{2} \le  q_i <  k_i+t+\frac{1}{2}$ or $ \lambda k_i+\lambda (t-1/2) \le  \lambda q_i <  \lambda k_i+\lambda (t+1/2).$ This implies $\lambda k_i-\frac{1}{2} < \lambda q_i < \lambda k_i+t+ \frac{1}{2},$ unless  $\lambda \le 2t+1$, in which case $\lambda q_i$ is bounded by $\lambda k_i+1/2$. In the former case, $\lambda k_i \le [\lambda q_i] \le \lambda k_i+t$; in the latter case,  $[\lambda q_i] = \lambda k_i$. The proof of the second part of (ii) is similar. Finally, statement (iii) follows from (ii).
\end{proof}

An apportionment method is said to {\it satisfy quota} if,  for each $i$, $h_i \in \{ \lfloor q_i \rfloor,  \lceil q_i \rceil\}$.
\begin{lemma}\label{sqrounding}
If $F$ satisfies quota, $F(\mathbf v,  H) =  \mathbf h$, and for some $\lambda<1$,  the $ \lambda h_i$ are integers, then $q_i$ is an integer if and only if $\lambda q_i$ is an integer.
\end{lemma}
\begin{proof}
If $q_i$ is an integer then $\lfloor  q_i \rfloor = \lceil  q_i \rceil$ and, since $F$ satisfies quota,  $\lfloor  q_i \rfloor = \lceil  q_i \rceil =  h_i$. From  Lemma \ref{quotarounding}(i), this implies $ \lambda h_i = \lfloor  \lambda q_i \rfloor = \lceil  \lambda q_i \rceil$, so that $\lambda q_i$ is an integer.

Alternatively, if $q_i$ is not an integer, then either $\lfloor  q_i \rfloor= h_i$ or $\lceil  q_i \rceil= h_i$. Without lost of generality, assume  $\lfloor  q_i \rfloor= h_i$. Hence, $h_i < q_i < h_i+1$ and so   $ \lambda h_i < \lambda q_i < \lambda h_i+\lambda  < h_i+1.$ Thus $ \lambda q_i$ is not an integer.
\end{proof}

To investigate which quota methods satisfy proportional consistency,  suppose  $F$ satisfies quota and let  $L_{(\mathbf v,  H)} = \{i \mid  h_i = \lfloor  q_i \rfloor\}$ and $U_{(\mathbf v,  H)} = \{i  \notin L_{( \mathbf v, H)} \mid   h_i = \lceil   q_i  \rceil\}$ be the set of states who receive their lower quota or upper quotas, respectively. Note that  $L_{ (\mathbf v, H)}$  contains those states whose quotas are integers so that $L_{(\mathbf v, H)}$ and $U_{(\mathbf v, H)}$ partition the set of states. Define $L_{(\mathbf v, \lambda H)}$ and $U_{(\mathbf v, \lambda H)}$ analogously. 

Suppose $F(\mathbf v, H) = \mathbf h$ and the $\lambda h_i$ are integers. By Lemma \ref{quotarounding},  $i \in  L_{(\mathbf v,  H)}$ implies  $\lfloor  q_i \rfloor =  h_i$ and $\lfloor  \lambda q_i \rfloor =  \lambda h_i$. Similarly,   $i \in  U_{(\mathbf v,  H)}$ implies $\lceil  q_i \rceil =  h_i$ and $\lceil  \lambda q_i \rceil  = \lambda h_i$. Since $\lambda q_i$ is not an integer, this implies $\lfloor \lambda q_i \rfloor = \lambda h_i-1$. So
 $$\sum_i \lfloor \lambda q_i \rfloor =\sum_{i \in L_{(\mathbf v,  H)}}  \!\! \lfloor \lambda q_i \rfloor+ \sum_{i \in U_{(\mathbf v,  H)}} \!\! \lfloor \lambda q_i \rfloor=  \sum_{i \in L_{(\mathbf v,  H)}} \!\! \lambda h_i+ \sum_{i \in U_{(\mathbf v,  H)}} (\lambda h_i-1) =H- \vert U_{(\mathbf v,  H)} \vert.$$
But, if $F(\mathbf v, \lambda H) = \mathbf h^\prime$, then
$$\sum_i \lfloor \lambda q_i \rfloor = \!\!\! \sum_{i \in L_{(\mathbf v, \lambda H)}} \!\! \lfloor \lambda q_i \rfloor + \!\!\! \sum_{i \in U_{(\mathbf v, \lambda H)}} \!\! \lfloor \lambda q_i \rfloor =  \!\!\! \sum_{i \in L_{(\mathbf v, \lambda  H)}} \!\! h_i^\prime + \!\!\! \sum_{i \in U_{(\mathbf v, \lambda H) }} \!\! (h_i^\prime-1) = H -  \vert U_{(\mathbf v, \lambda H) }  \vert,$$
so that $\vert U_{(\mathbf v, \lambda H)}  \vert =  \vert U_{(\mathbf v, H)}  \vert$ and $\vert L_{(\mathbf v, \lambda H)}  \vert =  \vert L_{(\mathbf v, H)}  \vert$. Thus if a method satisfies quota, the same number of states receive their quota rounded up (down) after scaling by $\lambda$.

\begin{proposition}\label{quotamethod}
If $F$ satisfies quota then

(i) if $F$ is proportionally consistent, $F(\mathbf v,  H) = \mathbf h,$  and for some $\lambda <1$  the $\lambda h_i$ are integers, then  $L_{( \mathbf v, H)} = L_{(\mathbf v, \lambda H)}$ and $U_{( \mathbf v, \lambda H)} = U_{(\mathbf v, H)}$; and

(ii)  if, for every $\mathbf v$ and $H$, where $F(\mathbf v, \lambda H) =\lambda \mathbf h$ and the $\lambda h_i$ are integers  for some $\lambda <1$,   $L_{( \mathbf v, \lambda H)} = L_{(\mathbf v, H)}$ and $U_{( \mathbf v, \lambda H)} = U_{(\mathbf v, H)}$, then  $F$ is proportionally consistent.
\end{proposition}

\begin{proof}
(i) Let $F$ satisfy the hypotheses and $F(\mathbf v, \lambda H) = \lambda \mathbf h$.  By Lemma \ref{quotarounding},  if $i \in L_{(\mathbf v,  H)}$,  then  $\lfloor  q_i \rfloor=h_i$ and $\lfloor \lambda q_i \rfloor = \lambda h_i$ so that  $i \in L_{(\mathbf v,\lambda H)}$.  Likewise, if  $i \in U_{(\mathbf v,  H)}$ then $\lceil q_i \rceil = h_i$ and $\lceil \lambda q_i \rceil = \lambda h_i$. As $q_i$ is not an integer, neither is $\lambda q_i$. Hence $i \in U_{( \mathbf v,\lambda H)}$. 

(ii) Let $F(\mathbf v,  H) = \mathbf h$, the $\lambda h_i$ be integers and $L_{(\mathbf v, \lambda H)} = L_{(\mathbf v, H)}$ and $U_{(\mathbf v, \lambda H)} = U_{(\mathbf v,H)}$.
Let $F(\mathbf v, \lambda H) = \mathbf h^\prime.$   If $i \in L_{(\mathbf v,  H)}$, then $\lambda h_i=\lfloor \lambda q_i \rfloor = h_i^\prime,$ since $i \in L_{(\mathbf v,\lambda H)}$.  Likewise, if $i \in U_{(\mathbf v,  H)} $, then $\lambda h_i=\lceil \lambda q_i \rceil = h_i^\prime,$ since $i \in U_{(\mathbf v,\lambda H)}$. Hence $\mathbf h^\prime =\lambda \mathbf h.$  \end{proof}

Prop. \ref{quotamethod} implies that, for any method  satisfying quota, proportional consistency is equivalent to the condition that the same states   receive their upper (or lower) quota after scaling by $\lambda <1$.   This property is true for a number of quota-based methods, including shift-quota methods. Prop. \ref{quotarounding} and Prop. \ref{quotamethod} remain true if  all references to quota and $q_i$ are replaced with the shifted quota and $q_i^s$. To show these methods are proportional, it suffices to show that the same states receive their shifted quota rounded up or down after scaling.  For $s \in [0, 1),$ let $L^s_{(\mathbf v, \lambda H)}$ and $U^s_{(\mathbf v, \lambda H)}$ denote the set of states who receive their lower (resp. upper) shifted quota. 

\begin{proposition} The shift-quota methods are proportionally consistent.
\end{proposition}
\begin{proof}  
If $i \in L^s_{(\mathbf v,  H)}$ and $j \in U^s_{(\mathbf v,  H)}$, then $ q_i^s - \lfloor  q_i^s \rfloor \le  q_j^s - \lfloor  q^s_j \rfloor$ or
$ q^s_i - h_i= q^s_i - \lfloor  q^s_i \rfloor   \le  q^s_j - \lfloor  q^s_j \rfloor=  q^s_j - ( h_j -1).$
Multiplying by $\lambda$ yields   $ \lambda q^s_i - \lambda h_i \le  \lambda q^s_j -  \lambda (h_j-1) \le \lambda q^s_j -  (\lambda h_j-1)$ or $\lambda q^s_i - \lfloor \lambda q^s_i \rfloor \le \lambda q^s_j - \lfloor \lambda q^s_j \rfloor$.
 This implies that after rounding down, the remaining seats go the same states, and hence the shift-quota methods are proportionally consistent by Prop. \ref{quotamethod}. \end{proof}

While Hamilton's method is the best known apportionment  method satisfying quota, Prop. \ref{quotamethod} suggests that any quota method which first awards $\lfloor q_i \rfloor$ seats and then allocates the remaining seats in a consistent way will be proportional.  This leads to consideration of methods such as the following.

\begin{definition} Let each state receive an initial number of seats equal to $\lfloor  q_i \rfloor$.  Under the Large (LAR) method, the remaining seats are given one each to the  $H - \sum_i \lfloor  q_i \rfloor$ states with the largest populations whose quotas are not integers.  Under the Small (SML) method, the remaining seats are given one each to the  $H - \sum_i \lfloor  q_i \rfloor$ states with the smallest populations whose quotas are not integers. \end{definition}

\begin{proposition} \label{four} The LAR and SML methods are  proportional. 
\end{proposition}

The proof of Prop. \ref{four} follows directly from Lemma \ref{quotarounding} and Prop. \ref{quotamethod}. The LAR and SML methods are clearly biased in the sense that they reward large or small states disproportionately. While such methods would not be appropriate for awarding seats to states in the US House, they may be relevant in other settings. The LAR method, for instance, is used by the US Republican Party in several states to award delegates to presidential candidates based on the candidates' share of the primary vote. In this context, it is considered advantageous to use an apportionment method that is slightly biased towards stronger candidates; see \cite{JMW}.  Rather than using population to select which states receive an additional seat, an alternate method that is also proportional (but not anonymous) is to designate a priority order of which states with non-integer quotas should receive one of the extra $H - \sum_i \lfloor  q_i \rfloor$ seats.  We turn to other types of quota methods.

Many quota-based methods can be defined by first rounding quotas up, down or to the nearest integer and then adjusting in ways that may not satisfy quota.  We consider several of these methods, determining which are proportionally consistent.

The Lower Quota Extremes (LQE) method takes the bias of the LAR method to the extreme. 

\begin{definition} Let each state receive an initial number of seats equal to $\lfloor  q_i \rfloor$. Under the Lower Quota Extremes method,  the remaining  $H - \sum_i \lfloor  q_i \rfloor$ seats are given to the largest state. \end{definition}

While the LQE method is even more biased than the LAR and SML methods, it is straightforward to prove the following. 

\begin{proposition} The LQE method is proportionally consistent.
\end{proposition}

\begin{proof} Suppose $F(\mathbf v H) =  \mathbf h$ and for some $\lambda < 1$ the $\lambda h_i$ are integers. Let  $F(\mathbf v \lambda H)  = \mathbf h^\prime$. Under LQE, $h_i = \lfloor  q_i \rfloor$ for all $i >1$. By Lemma \ref{quotarounding}, this implies $\lfloor \lambda q_i \rfloor = \lambda h_i=h_i^\prime$ for all $i>1$, and hence $h_1^\prime = \lambda H - \sum_{i >1} \lambda h_i= \lambda h_1$, which implies $\mathbf h^\prime=\lambda \mathbf h.$
\end{proof}

Similar arguments apply to methods based on upper quota rounding in which  each state initially receives  $\lceil q_i \rceil$ seats and then   $\sum_i \lceil q_i \rceil-H$ seats are removed in a systematic way.  One such example is the Sequential Upper Quota method, which, like LAR, is used by some Republican state parties to allocate delegates to candidates during the US presidential primaries; see \cite{JMW} for which states use this method.

\begin{definition} The Sequential Upper Quota (SUQ)  method initially assigns each state $\lceil q_i \rceil$ seats.   The $\sum_i \lceil q_i \rceil- H$ over-allocated seats are removed from the smallest state; if the smallest state does not have a sufficient number of seats, the rest are removed from the second-smallest state, and so on, until the house size is met.
\end{definition}

The SUQ method can alternatively be described as giving $\lceil q_i \rceil$ seats to states in descending order of size until no more seats are available.   

\begin{proposition} The SUQ method is proportionally consistent.
\end{proposition}
\begin{proof} Let $F(\mathbf v,  H) =  \mathbf h$ and $\lambda <1$ be such that the $\lambda h_i$ are integers. Let  $F(\mathbf v, \lambda H)= \mathbf h^\prime$.
Assume $\mathbf h=( \lceil q_1 \rceil, \lceil q_2 \rceil, \ldots, \lceil q_{k-1} \rceil, H-\sum_{i<k} \lceil q_i \rceil, 0, \ldots, 0)$.
  If $ i >k$ then $h_i=0$. If $i < k$, then $\lceil  q_i \rceil =  h_i$ which, by Lemma \ref{quotarounding}, implies   $\lceil \lambda q_i \rceil = \lambda h_i$.   If   $\lceil  q_{k} \rceil= h_{k}$, then   $ \lceil \lambda q_k \rceil =\lambda h_k$; otherwise  $h_{k}< \lceil  q_{k} \rceil$ which implies, since $h_{k}$ is an integer, $h_{k} \le \lfloor  q_k \rfloor \le  q_k$, or  $\lambda h_{k} \le \lambda q_i \le \lceil \lambda q_k \rceil $. In either case, $ \lceil \lambda q_k \rceil \ge \lambda h_k$,  so
$$\sum_{i < k} \lceil \lambda q_{i} \rceil = \sum_{i < k} \lambda h_i \le \lambda H  = \sum_{i < k}\lambda h_i + \lambda h_{k} \le \sum_{i \le k}\lceil \lambda q_{i} \rceil.$$ 
Hence when apportioning $\lambda H$ seats,  the largest $k-1$ states receive $\lceil \lambda q_i \rceil$ seats, with the $k{\text{th}}$-largest state receiving the remaining seats.  So
\begin{equation*}
\mathbf h^\prime= (\lceil \lambda q_1 \rceil, \ldots, \lceil \lambda q_{k-1} \rceil, \lambda H-\sum_{i<k} \lceil \lambda q_1 \rceil, 0, \ldots, 0)
 = (\lambda h_1, \ldots \lambda h_{k-1}, h_k^\prime, 0, \ldots, 0)
\end{equation*}
where $h_k^\prime = \lambda H - \sum_{i \ne k}\lambda h_i = \lambda h_k$. Thus, SUQ is proportionally consistent. \end{proof}

Other apportionment methods are most easily described when quotas are rounded to the nearest integer and then allocations are adjusted. Note that since $q_i -1/2 < [q_i] \le q_i+ 1/2$, then  $H-n/2 < \sum_i [q_i]  \le H + n/2$. The upper bound is strict unless $n$ is even and $[q_i]=q_i+1/2$  (so the fractional part of the quota is equal to $0.5$) for all $i$. Thus, apart for this rare exception, $\vert \sum_i [q_i] - H \vert \le \lfloor (n-1)/2 \rfloor,$ and at most $\lfloor (n-1)/2 \rfloor$ seats can be under or over-allocated.

For comparison, we introduce two methods: Nearest Integer Extremes (NIE) and Nearest Integer Sequential (NIS), which are both used in the  US Republican Party presidential primary \cite{JMW}. These are similar to the LQE and LAR methods respectively: the NIE method corrects for under- or over-allocation by adjusting the apportionment of the largest or smallest state; the NIS method corrects  for under- or over-allocation  by adjusting the allocations of multiple large or small states.

\begin{definition} Under the Nearest-Integer Extremes (NIE) method,  each state initially receives $[q_i]$ seats. If $H \ge \sum_i [q_i]$, the remaining $c=H-\sum_i [q_i]$ seats are  given to state 1.  If  $H < \sum_i [q_i]$, the extra  $c = \sum_i [q_i]-H$ seats are removed from state $n$; if state $n$ does not have enough seats, additional seats are taken from state $n-1$, etc. \end{definition}

\begin{definition}
Under the Nearest-Integer Sequential (NIS) method,  each state initially receives $[q_i]$ seats. If $H \ge \sum_i [q_i]$, the remaining $c= H - \sum_i [q_i]$ seats are  given one each to states $i \le c$.   If  $H < \sum_i [q_i]$,  the extra $c=\sum_i [q_i] - H$ seats are removed one each from the smallest states. \end{definition}

To investigate proportional consistency, we require the following lemma, whose proof follows from Lemma \ref{quotarounding}.

\begin{lemma}\label{rounding}
\label{NISbounds2} Suppose $F(\mathbf v,  H) =  \mathbf h$ and for some $\lambda <1$, the  $\lambda h_i$ are all integers.
\vspace{-2mm}
\begin{enumerate}[label=(\roman*)]
\item  If $[q_i] =  h_i$ for all $i$, then  $\sum_i [\lambda q_i] = \lambda H$.

\item If $\sum_i [q_i] =  H + c$ for some $c >0$, then  $\lambda H \le \sum_i [\lambda q_i]\le \lambda H+ c$. If $\sum_i [q_i] =  H-c$,  then $\lambda H-c \le  \sum_i [\lambda q_i]  \le \lambda H.$  
In both cases,  if $\lambda \le 1/(2k+1)$  where $k$ satisfies $k \ge \max_i \vert [q_i] - h_i \vert$, then $\sum_i [\lambda q_i] = \lambda H$. 
\end{enumerate}
\end{lemma}

Lemma \ref{rounding} shows that any failures of proportional consistency for the NIE or NIS methods occur    when $\sum_i [ q_i] \ne  H$ or when $\lambda \ge 2k+1$.  Under the NIS method, $[q_i] =  h_i \pm 1$, so,  Lemma \ref{rounding}  implies $F(\mathbf v, \lambda H) = \lambda \mathbf h$ for all $\lambda \le  1/(2(1)+1)=1/3$. Under the NIE method, as discussed previously, unless $n$ is even and the fractional parts of each quota is $1/2$, 
$\vert \sum_i [q_i] -H\vert \le \lfloor (n-1)/2 \rfloor$ so $F(\mathbf v, \lambda H) = \lambda \mathbf h$ for all $\lambda \le 1/(2(\lfloor (n-1)/2 \rfloor +1)= 1/n$ if $n$ is odd and $\lambda \le 1/(n-1)$ if $n$ is even.

The NIE method strays further from proportionality than the NIS method. However, the NIE method  is proportionally consistent while the NIS method is not.

\begin{proposition} \label{prop7} The NIE method is proportionally consistent.
\end{proposition}

\begin{proof} 
Suppose $F(\mathbf v,  H) =  \mathbf h$ and   $F(\mathbf v, \lambda H) =  \mathbf h^\prime$.   If $\sum_i [q_i] =  H -c$ for some $c \ge 1$, then $\mathbf h = ([ q_1]+c, [ q_2], \ldots, [ q_n])$.

 If $i \ge 2$,  then $[ q_i]=  h_i $ which implies  $[\lambda q_i] = \lambda h_i$.    In addition,   $[ q_1] =  h_1-c$,  which implies  $ q_1 < h_1 - c+\frac{1}{2}$ or  $\lambda q_1 < \lambda h_1 -\lambda (c-1/2) <\lambda h_1+ \frac{\lambda}{2}$.  Thus $[\lambda q_1] \le \lambda h_1,$ which means $\sum_i [\lambda q_i] \le \lambda H$. So $h_i^\prime =  [\lambda q_i]=\lambda h_i $ for $i \ge 2$ and $h_1^\prime = \lambda H - \sum_{i \ge 2} \lambda h_i =\lambda  h_1$, as required.

If $\sum_i [ q_i] =  H+c$ for some $c \ge 1$ then $\mathbf h = ([q_1], [q_2], \ldots, [q_{k-1}], h_k, 0, \ldots, 0)$ for some $k$ where $h_k = H-\sum_{i<k}[q_i] \le [q_k]$. If $i < k$ then $[q_i]=  h_i$ implies $[\lambda q_i] = \lambda h_i$.  If $i \ge k$ then $[q_i] \ge  h_i$ which implies $[\lambda q_i] \ge \lambda h_i$. Thus $\sum_i [q\lambda _i] \ge \lambda H,$ and hence $\mathbf h^\prime = (\lambda h_1, \ldots,  \lambda h_{k^\prime - 1}, [\lambda q _{k^\prime}]-c^\prime, 0,  \ldots, 0)$
 for some $k^\prime \ge 1$ and $0 \le c^\prime < [\lambda q_{k^\prime}]$.  

We claim $k^\prime = k$.  To see this, suppose $k^\prime <k$ so that $\mathbf h^\prime$ has more $0$ entries then $\mathbf h$. But $h^\prime_i = [ \lambda q_i] = \lambda h_i$ for all $i \le k^\prime-1$ and   $h^\prime_{k^\prime}  < [\lambda q_{k^\prime}] = \lambda h_{k^\prime}$. So $\lambda H = \sum_{i \le k^\prime}\lambda h_i^\prime < \sum_{i \le k^\prime}\lambda h_i < \lambda  H$, which is a contradiction. The case $k^\prime > k$ is shown analogously.  

Thus   $\mathbf h^\prime = (\lambda h_1, \ldots,  \lambda h_{k - 1}, [\lambda q _{k}]-c^\prime, 0,  \ldots, 0),$ which implies  $[\lambda  q_{k}]-c^\prime = \lambda h_k$ and $\mathbf h^\prime =\lambda \mathbf h$.
\end{proof}

\begin{proposition}\label{NISnot} The NIS method is not proportionally consistent.
\end{proposition}
\begin{proof}
Let $n=8$, $H = 70$, and $\mathbf v =(1000, 965, 965, 965, 965, 965, 625, 550)$. Then, $F(\mathbf v, 70) = (10, 10, 10, 10, 10, 10, 5, 5)$.  If $\lambda = \frac{2}{5}$, then $\lambda H = 28$ and $F(\mathbf v, 28) = (4, 4, 4, 4, 4, 4, 3, 1)$, which is not equal to $\frac{2}{5} F(\mathbf v, 70) = (4, 4, 4, 4, 4, 4, 2, 2)$. \end{proof}

While the NIS method fails proportional consistency generally, it does satisfy proportional consistency for $n \le 4$ and also in part for $n=5$ (when $\lambda \le 1/2$).

\begin{proposition} \label{prop9} Under NIS, if  $F(\mathbf v  H) =  \mathbf h$ and there exists a $\lambda <1$ such that the $\lambda h_i$ are integers, then $F(\mathbf v,  \lambda H) =  \lambda \mathbf h$   if: (i) $n \le 4$; or  (ii) $n=5$ and $ \lambda \le 1/2$.  
\end{proposition}

\begin{proof}
If $n \le 4$, when allocating $H$ seats, there is at most 1 seat under or over-allocated initially,  except if $n=4$ and 2 seats are over-allocated. This occurs only if  all $q_i$ have fractional part $0.5$.  If $i=1, 2$,  $ h_i=[ q_i]$, which implies $\lambda h_i = \lambda q_i$. And if $i=3,4$, then $ h_i = [ q_i]-1 = \lambda q_i-1/2$ or $\lambda h_i=\lambda q_i - \lambda/2.$ Since $\lambda <1$, this implies $[\lambda q_i]=\lambda h_i$. Hence $\sum_i [\lambda q_i] = \lambda H$ and so $h_i^\prime = [\lambda q_i] = \lambda h_i$ for all $i$.
In all other cases, only 1 seat is under or over-allocated initially so that the NIS and  NIE methods coincide.  Since the NIE method is proportional, so is the NIS allocation.

If $\frac{1}{3} \le \lambda < \frac{1}{2}$ and $n=5$, there are at most 2 seats under or over allocated initially. If 1 seat is over or under allocated then NIS and NIE agree with $F(\mathbf v  \lambda H) = \lambda \mathbf h$.  Suppose  2 seats are under allocated. Then  $h_i = [q_i]$ (implying $\lambda h_i = [\lambda q_i]$) if $i =3, 4, 5$ and $h_i = [\mathbf q_i] + 1$ if $i = 1, 2$.  Hence,
 $\sum_i [\lambda q_i] \ge \lambda H-2$, $[\lambda q_i] = \lambda h_i$ for $i = 3, 4, 5$ and $[\lambda q_i] \in \{\lambda h_i-1, \lambda h_i\}$ for $i=1,2.$  

We claim that $[\lambda q_2]=\lambda h_2$. If this is the case, then $\sum_i [\lambda q_i] \ge \lambda H-1,$ so there is at most one seat unallocated, which must go to state 1. Hence  $[\lambda q_i]=\lambda h_i$ for all $i$ and $F(\mathbf v  \lambda H) = \lambda \mathbf h$.

To prove the claim, suppose  $[\lambda q_2]=\lambda h_2-1$.  Then $\lambda q_2 <\lambda h_2-1/2$, and hence $q_2 <  h_2-1/2\lambda $. If $\lambda \le 1/2$,then $q_2 <   h_2-1= [q_2].$ Thus  $[q_2] = \lfloor   q_2 \rfloor+1.$ For each $i$, let $q_i = \lfloor  q_i \rfloor+x_i$ for some $0 \le x_i < 1$ with  $\sum_i x_i \le n-1=4$.
Then
$$ \sum_{i} \lfloor  q_i \rfloor + \sum_i x_i =  H = \sum_i[ q_i]+2 \text{ and }
\lfloor  q_2 \rfloor + \sum_{i \ne 2} \lfloor  q_i \rfloor + \sum_i x_i =\lfloor  q_2 \rfloor+1+\sum_{i \ne 2} [ q_i] + 2.$$
Thus $\sum_i x_i  = \sum_{i \ne 2}( [ q_i]  -  \lfloor  q_i \rfloor ) + 3 \ge 3,$ which  implies $\sum_i x_i \in \{3, 4\}.$
 If  $\sum_{i} x_i = 3$, then $ [q_i]  =  \lfloor  q_i \rfloor$ for $i \ne 2$.  But then  $\sum_{i \ne 2} x_i < 4(0.5) = 2$, which implies  $x_2 > 1$, leading to a contradiction. Similarly, if  $\sum_{i} x_i = 4$, then $ [q_i]  =  \lfloor  q_i \rfloor$ for all $i \ne 2$ except for one state, say $k$.    Then $\sum_{i \ne 2,k} x_i < 3(0.5) = 1.5$, and so $x_2+x_k > 2.5$, which is another contradiction. This proves the claim.
 
The case when 2 seats are over-allocated is analogous. \end{proof}

When $\frac{1}{2} < \lambda < 1$ and $n=5$, it is not necessarily true that $F(\mathbf v, \lambda H) = \lambda \mathbf h$.  For a counterexample, let $H=40$, $\lambda = \frac{4}{5}$,  $\mathbf p  = \frac{1}{75}(14.375, 9.35, 5.425, 5.425, 5.425)$. Then  $F(\mathbf v, 40)=  (15, 10, 5, 5, 5)$ however, $F(\mathbf v, \lambda 40)= (13, 7, 4, 4, 4)  \ne   (12, 8, 4, 4, 4)$.

\section{Stationary Quotatone Methods} \label{quotatone}

One of the central  problems in apportionment theory is   the  incompatibility of satisfying  quota (as does Hamilton's method) and satisfying monotonicity properties  (as do divisor methods).  In \cite{BYQ}, Balinski and Young  define the Quota method which satisfies  both house monotonicity and  quota.  Subsequently, they identified a more general family of apportionment methods---the quotatone methods---which they  show uniquely satisfy both properties.  

Quotatone methods are defined inductively, in a manner analogous to divisor methods which can be described equivalently  by assuming that if $H$ seats are allocated by $F(\mathbf v, H) = \mathbf h$,  then the $(H+1){\text{st}}$ seat is given to the state satisfying $\max_i \frac{v_i}{f(h_i)}.$  Quotatone methods require also  that the  $(H+1){\text{st}}$ seat be given to an   {\it eligible} state, where a state is eligible if awarding it the next seat would not:  (i) break upper quota; or (ii)  break lower quota in the next several induction steps.

To define eligibility more precisely,  suppose  that $H$ seats have been apportioned with  $F(\mathbf v, H) = \mathbf h$.  Let $U(\mathbf v, \mathbf h) = \{i \mid h_i < (H+1)\frac{v_i}{V}\}$.  If $i \notin U(\mathbf v, \mathbf h)$, it follows that $h_i \ge \lceil (H+1)\frac{v_i}{V} \rceil$ or $h_i+1>(H+1)\frac{v_i}{V},$ which implies giving $S_i$ seat $H+1$ would break upper quota.  Hence, eligible states  must lie in $U(\mathbf v, \mathbf h)$.

To define conditions for not breaking lower quota,  for each integer $\alpha \ge 1,$ let $L_\alpha (\mathbf v, \mathbf h) = \{ i \mid \lfloor (H+\alpha)\frac{v_i}{V} \rfloor - h_i \ge 1\}$ and let $g(\alpha) =  \sum_{i \in L_\alpha (\mathbf v, \mathbf h)} ( \lfloor (H+\alpha) \frac{v_i}{V}\rfloor - h_i)$.  

We consider what it means for $g(\alpha) \ge \alpha$. $L_1 (\mathbf v, \mathbf h) = \{ i \mid \lfloor (H+1)\frac{v_i}{V} \rfloor - h_i \ge 1\}$ consists of all states that must receive a seat in the next step  to not break lower quota. If $g(1) \ge 1$,  there is at least one state in $ L_1 (\mathbf v, \mathbf h)$ who must be given the next seat.  $L_2 (\mathbf v, \mathbf h) = \{ i \mid \lfloor (H+2)\frac{v_i}{V} \rfloor - h_i \ge 1\}$ consists of all states that must receive another seat in the next two steps to not break lower quota.  If $g(2) \ge 2$, at least one state in $ L_2 (\mathbf v, \mathbf h)$  must be given the next seat. This is because $L_2 (\mathbf v, \mathbf h)$ must contain either: two states  for whom  $\lfloor (H+2)\frac{v_i}{V} \rfloor - h_i \ge 1$, who must get 1 more seat in the next two steps; or one state  for whom $\lfloor (H+2) \frac{v_i}{V}\rfloor - h_i \ge 2$, who must get 2 more seats in the next two steps. Similarly, if $g(\alpha) \ge \alpha$,  the next seat must be given to a state in  $L_\alpha(\mathbf v, \mathbf h)$ to not break  lower quota in the next $\alpha$ steps.

 Let
 $\tilde{\alpha} = \min \{\alpha \ge 1 \mid g(\alpha) \ge \alpha \}$ and let
 $L(\mathbf v, \mathbf h) = L_{\tilde{\alpha}}(\mathbf v, \mathbf h)$ or, if no such $\tilde{\alpha}$ exists, let  $L(\mathbf v, \mathbf h) =\{1, 2, \ldots, n\}$.  Combining the criteria, the set of eligible  states is  $L(\mathbf v, \mathbf h) \cap U(\mathbf v, \mathbf h).$  In \cite{BY}, Balinski and Young  prove:  (i) $L(\mathbf v, \mathbf h) \cap U(\mathbf v, \mathbf h) \ne \emptyset$; and (ii) if $\tilde{\alpha}$ exists, it satisfies $\tilde{\alpha} \le \max_i \lceil \frac{h_i-H(v_i/V)}{v_i/V} \rceil$. They show the following.

\begin{proposition}(Balinski and Young \cite{BY})
A house monotonic method $F$ satisfies quota if and only it can be defined inductively as follows. For each $\mathbf v$, let $F(\mathbf v, 0) = \mathbf 0$ and, assuming   $\mathbf F(\mathbf v, H)=\mathbf h$, give  seat  $H+1$  to some state   $i \in L(\mathbf v, \mathbf h) \cap U(\mathbf v, \mathbf h)$.  \end{proposition}

Chief among quotatone methods are those based  on  divisor methods.
\begin{definition}
Given a rounding rule $f$, its associated quotatone apportionment method assigns seat $H+1$ to the state in $L(\mathbf v, \mathbf h) \cap U(\mathbf v, \mathbf h)$ that maximizes $\frac{v_i}{f(h_i)}.$
\end{definition}

The next two propositions show that quotatone methods based on stationary divisor methods and Hill-Huntington's method are not  proportionally consistent.  The proof of the first proposition requires a lemma.

\begin{proposition}\label{quota1}
Suppose $f(k)= k+s$ is a rounding rule for some $s \in [0.5,1]$ and let  $F$ be the corresponding quotatone method. If  
$\mathbf v =(48569, 41012, 8200, 1115, 1095),$ then $F(\mathbf v, 40) = (20, 16, 4, 0,0)$ but  $F(\mathbf v, 30) = (15, 13, 2, 0,0) \ne \frac{3}{4}F(\mathbf v, 40)$.  
\end{proposition}

 \begin{lemma}\label{limit1}
Let $F$ and $\mathbf v$ be as in Prop. \ref{quota1} with $F(\mathbf v, H) = \mathbf h$.   Then $h_4 = h_5 = 0$ for all $H \le 36$.  
 \end{lemma}

 \begin{proof} Let $\mathbf p = \frac{\mathbf v}{V}$ be the population distribution.
Let  $\hat{H}$ be the largest house size such that $S_4$ receives no seats, and let $F(\mathbf v, \hat{H}) = \hat{\mathbf h}$ with $\hat{h}_4=0.$ Then $S_4$ receives seat $\hat{H}+1$  and hence $4 \in U(\mathbf v, \hat{\mathbf h}) \cap L(\mathbf v, \hat{\mathbf h}).$  

Suppose $\hat{H} \le 36$. Direct calculation shows $\lceil 36\hat{p}_1 \rceil \le 18$,  $\lceil 36\hat{p}_2 \rceil \le 18$ and $\lceil 36\hat{p}_3 \rceil \le 3,$ so $\hat{h}_1 \le 18$, $\hat{h}_2 \le 18$ and $\hat{h}_3 \le 3$. It is easy to see that, for all  $s \in [0.5, 1]$,  $\frac{p_1}{\hat{h}_1+s}> \frac{p_4}{s}$, $\frac{p_2}{\hat{h}_2+s}> \frac{p_4}{s}$ and $\frac{p_3}{\hat{h}_3+s}> \frac{p_4}{s}$. Since $S_4$ receives the seat, this implies that  $S_1, S_2$ and $S_3$ are not eligible. 

We claim   $\{1, 2, 3\} \cap U(\mathbf v, \hat{\mathbf h})  \ne \emptyset.$ If not then,  $\hat{h}_i \ge (\hat{H}+1)p_i$ for $i =1, 2, 3$. Summing these inequalities  yields  $\hat{H} \ge (\hat{H}+1)(p_1+p_2+p_3)$ or  $\hat{H} \ge \frac{p_1+p_2+p_3}{p_4+p_5} \ge 44.2$, which is a contradiction.  Hence, there exists $j \le 3$ such that $j \in U(\mathbf v, \hat{\mathbf h}).$

 Next, we claim $L(\mathbf v, \hat{\mathbf h})=  \{1, 2, 3, 4, 5\}$. If not, then  $L(\mathbf v, \hat{\mathbf h}) = L_{\tilde{\alpha}}(\mathbf v,\hat{\mathbf h})$ for some $1 \le \tilde{\alpha} \le \max_i \lceil \frac{\hat{h}_i-\hat{H} p_i}{p_i} \rceil.$  Since $\hat{\mathbf h}$ satisfies quota, $\hat{h}_i-\hat{H}p_i \le 1$ for all $i$ and $\hat{h}_i-\hat{H}p_i <0$ for any state receiving its lower quota. Thus  $\tilde{\alpha} \le  \max_{i \le 3}\lceil \frac{\hat{h}_i-\hat{H}p_i}{p_i} \rceil   \le \max_{i \le 3} \frac{1}{p_i} <13$, which implies  $\hat{H}+ \tilde{\alpha} \le 36+12 = 48.$ Since $48p_4 < 1$,  $\lfloor (\hat{H}+\tilde{\alpha})p_4 \rfloor =0$ and hence $4 \ne L_{\tilde{\alpha}}(\mathbf v,\hat{\mathbf h}),$ which is a contradiction.
 
Thus  $L(\mathbf p,\hat{\mathbf h})=  \{1, 2, 3, 4, 5\}$, which implies  $j \in  U(\mathbf v, \mathbf h) \cap L(\mathbf v, \mathbf h)$ for some $j \le 3$, which is impossible.  So  $\hat{H} > 36$, which implies $h_4, h_5=0$ for all $H \le 36.$   \end{proof}

\begin{proof} (Prop. \ref{quota1}) If $H=24$, then $\lfloor \mathbf q \rfloor = (11, 9, 1, 0,0)$. Since $\sum_i \lfloor q_i \rfloor  =21$ and  $h_4 = h_5=0$ by Lemma \ref{limit1}, we have  $F(\mathbf v, 24) = (12, 10, 2,0,0).$  Next, since $\frac{p_1}{h_1+s} > \frac{p_3}{2+s}$ if $h_1 \le 14$ and $\frac{p_2}{h_2+s} > \frac{p_3}{2+s}$ if $h_2 \le 12$ for all $s \in [0.5,1]$, $S_3$ cannot receive additional seats until $h_1\ge 15$ and $h_2 \ge 13$.  
It easily follows that $F(\mathbf v, 28) =(14, 12, 2, 0,0)$,  $F(\mathbf v, 29)=(15, 12, 2, 0,0)$ and $F(\mathbf v, 30)=(15, 13, 2, 0,0)$.

Similarly, if $H=36$,  $\lfloor \mathbf q \rfloor =(17, 14, 2, 0, 0)$, implying $F(\mathbf v, 36)=(18, 15, 3,0,0)$.  For $H=37, 38, 39, 40$, the results are, respectively, the following: $(18, 16, 3,0,0);$ $(19, 16, 3,0,0)$; $(19, 16, 4,0,0)$; $(20, 16, 4,0,0)$.
\end{proof}

The next proposition is a counterexample when  $\lambda = 4/5$ for  quotatone methods where $f(k)=k+s$ for $s \in [0,0.5]$ and Hill-Huntington's method.  

\begin{proposition}\label{quota2}
Let $f(k)= k+s$ be a rounding rule for $s \in [0, 0.5]$ or $f(k) = \sqrt{k(k+1)}$ and $F$ be the associated quotatone method.
If $\mathbf v = (57535, 56825, 4027,$ $3318, 3295)$, then $F(\mathbf v, 175) = (80,80,5,5,5)$ but $F(\mathbf v, 140) = (64,63,5,4,4) \ne  \frac{4}{5} F(\mathbf v,175)$.\end{proposition}

We first prove a lemma before proving this proposition.

\begin{lemma}\label{limit2}
 Let $F$ and $\mathbf v$ be as in Prop. \ref{quota2} with $F(\mathbf v, H) = \mathbf h$. Then: (i)   $h_3 \le 4$ for all $H\le 137$; and (ii)  $h_3 \le 5$ for all $H \le 175$.    
 \end{lemma}

\begin{proof} Let $\mathbf p = \frac{\mathbf v}{V}$ be the population distribution.
To prove (i), note that   $Hp_3 < 4$ for all $H \le 124$. Thus, in order for  state 3 to receive a $5{th}$ seat for some house size $H \le 137,$ there must be an $\hat{H}$ in the interval $124 \le \hat{H} \le 136$ such that $F(\mathbf v, \hat{H}) =\hat{ \mathbf h}$ satisfies $\hat{h}_3=4$ and such that the  $(\hat{H}+1)\text{st}$ seat goes to state 3.

Suppose such an $\hat{H}$ exists. Since  $Hp_3<5$ and $3 < Hp_4, Hp_5< 4$ for all $H$ in this interval,  $\hat{h}_3 \in \{4, 5\}$ and $\hat{h}_4, \hat{h}_5 \in \{3, 4\}$.

Suppose $\hat{h}_5=3$.   Direct calculation shows that $\frac{p_5}{3+s} > \frac{p_3}{4+s}$ for all $s \in [0, 0.5]$ and $\frac{p_5}{\sqrt{12}} > \frac{p_3}{\sqrt{20}}$.   Since the next seat is given to state 3,  state 5 cannot be eligible. But $5 \in  U(\mathbf v, \hat{\mathbf h}),$  which implies $5 \notin L(\mathbf v, \hat{\mathbf h}).$ Thus, $\tilde{\alpha}$ exists and $3 \in L_{\tilde{\alpha}}(\mathbf v, \hat{\mathbf h})$, so $ \lfloor (\hat{H}+\tilde{\alpha}) p_3 \rfloor - 4 \ge 1$.  But $\lfloor 155p_3 \rfloor \le 4$, which implies $\hat{H}+\tilde{\alpha} \ge 156.$

 But   $\tilde{\alpha}\le \max_i \lceil  \frac{h_i-\hat{H}p_i}{p_i} \rceil  \le  \max \{ \lceil \frac{1}{p_1} \rceil, \lceil \frac{1}{p_2} \rceil, \lceil \frac{\hat{h}_4-\hat{H}p_4}{p_4} \rceil \} = \max \{3, \lceil \frac{\hat{h}_4-\hat{H}p_4}{p_4} \rceil\}.$
   If $\hat{h}_4=3$ then $\tilde{\alpha} \le 3$ and $\hat{H} + \tilde{\alpha} \le 139.$ If $\hat{h}_4=4$, then direct calculation shows $\tilde{\alpha} \le \{3, 151-\hat{H}\}$, and so $\hat{H}+\tilde{\alpha} \le 151.$ In either case, we have  a contradiction.

Thus, $\hat{h}_5=4,$ which implies $\hat{h}_4=4$. So, for some $124 \le \hat{H} \le 136,$ the apportionment is $(\hat{h}_1, \hat{h}_2, 4, 4, 4)$ and for $\hat{H}+1$, the apportionment is $(\hat{h}_1, \hat{h}_2, 5, 4, 4)$. But if  $H=135$, $\lfloor \mathbf q \rfloor =(62,61,4,3,3)$ which sums to 133,  allowing only two states to receive their upper quota. Thus,  $\hat{H}=135$ or $136$.   

Suppose $\hat{H}=135$ and $\hat{\mathbf h} = (62, 61, 4, 4, 4)$ and state 3 gets seat 136.
 Since $137p_1 > 63$ and $137p_2 > 62$,  the apportionment for $H=137$ must break quota for either state 1 or 2.  Alternatively, if $\hat{H}=136$ then either $\hat{\mathbf h} = (63, 61, 4, 4, 4)$ or $\hat{\mathbf h}=(62, 62, 4, 4, 4)$. Since state 3 gets the $137\text{th}$ seat, the apportionment for $H=137$ again again breaks either quota for state 1 or 2.
Thus, no such $\hat{H}$ exists, implying that $h_3 \le 4$ for all $H \le 137.$

The proof of (ii) is similar. Since  $Hp_3<5$ for $H \le 155$,  in order for state 3 to receive a $6\text{th}$ seat for some house size $H \le 175,$ there must be an $\hat{H}$ in the interval $155 \le \hat{H} \le 174$ such that $F(\mathbf v, \hat{H}) =\hat{ \mathbf h}$ satisfies $\hat{h}_3=5$ and such that the  $(\hat{H}+1)\text{st}$ seat goes to state 3.

Suppose such an $\hat{H}$ exists. Since $Hp_3 < 6$ and $4< Hp_4, Hp_5 <5$ for all $H$ in this interval,  $\hat{h}_3 \in \{5, 6\}$ and $\hat{h}_4, \hat{h}_5 \in \{4, 5\}.$

Suppose $h_5=4$. Then it is easy to see that  $\frac{p_5}{4+s} > \frac{p_3}{5+s}$ for all $s \in [0, 0.5]$ and   $\frac{p_5}{\sqrt{20}} > \frac{p_3}{\sqrt{30}}$.  Since the next seat is given to state 3, this implies state 5 cannot be eligible. But $5 \in  U(\mathbf v, \hat{\mathbf h}),$  which implies $5 \notin L(\mathbf v, \hat{\mathbf h}).$. So $\tilde{\alpha}$ exists  and $3 \in L_{\tilde{\alpha}}(\mathbf v, \hat{\mathbf h}),$ which implies  $ \lfloor (\hat{H}+\tilde{\alpha}) p_3 \rfloor - 5 \ge 1$.  Since $\lfloor 186p_3 \rfloor =5$,  $\hat{H}+\tilde{\alpha} \ge 187.$

But $\tilde{\alpha} \le \{ \frac{1}{p_1}, \frac{1}{p_2}, \lceil \frac{\hat{h}_4-\hat{H}p_4}{p_4} \rceil\} = \max\{3,  \lceil \frac{\hat{h}_4-\hat{H}p_4}{p_4} \rceil\} .$ If $\hat{h}_4=4$ then   $\tilde{\alpha} \le 3$ and so $\hat{H}+\tilde{\alpha} \le 175+3 = 178$, which is a contradiction. Thus $\hat{h}_4=5$, which implies $\tilde{\alpha}\le  \max \{ 3, \lceil \frac{5-\hat{H}p_4}{p_4} \rceil \} =\max\{3, 189-\hat{H}\}$, and so  $\hat{H} + \tilde{\alpha} \le 189.$ 

Thus, we must have $187 \le \hat{H}+\tilde{\alpha} \le 189.$ But direct calculation  shows that $\lfloor (\hat{H}+\tilde{\alpha}) p_1 \rfloor =86$, $\lfloor (\hat{H}+\tilde{\alpha}) p_2 \rfloor=85$,  $\lfloor (\hat{H}+\tilde{\alpha})p_3 \rfloor=6$, and $\lfloor (\hat{H}+\tilde{\alpha}) \hat{p}_4 \rfloor \in  \{4, 5\}$, which implies $4, 5 \notin L_{\tilde{\alpha}}(\mathbf v, \hat{\mathbf h})$. Hence  
 \begin{align*}
g(\tilde{\alpha}) &=  \sum_{i \in L_{\tilde{\alpha}}(\mathbf v, \hat{\mathbf h})} ( \lfloor (\hat{H}+\tilde{\alpha})p_i \rfloor - \hat{h}_i) \le \sum_{i=1,2,3} ( \lfloor (\hat{H}+\tilde{\alpha}) p_i \rfloor - \hat{h}_i)\\
 & =(86-\hat{h}_1)+(85-\hat{h}_2) + (6-\hat{h}_3)
 =177-(\hat{H}-\hat{h}_4-\hat{h}_5)  = 186-\hat{H}.
  \end{align*}  Hence    $g(\tilde{\alpha})  \le  186-\hat{H}<\alpha$, contradicting the definition of $L_{\tilde{\alpha}}(\mathbf v, \hat{\mathbf h}).$ Hence no such $\hat{H}$ exists, implying that $h_3 \le 5$ for all $H \le 175$.   
 \end{proof}

\begin{proof} (Prop. \ref{quota2}) If $H=137$, then $\lfloor \mathbf q \rfloor =(63, 62, 4, 3, 3)$, which sums to 135. By  Lemma \ref{limit2},  the only possible apportionments  for $H=137$ are $(64, 63, 4, 3,3)$,  $(64, 62, 4, 4,3)$,  $(63, 63, 4, 4,3)$ or $(63, 62, 4, 4, 4)$.  

We claim that the first two apportionments are impossible.  To see this, note that for $H=136$, the quotas are $(62.6, 61.8, 4.38, 3.61, 3.58)$, and so $\lceil 136p_1 \rceil = 63$ and  $\lceil 136p_2 \rceil = 62$.  In order for the apportionment at $137$ to be $(64, 63, 4, 3,3)$, both states 1 and 2 would have to receive seat  $137$.   If the apportionment at $H=137$ is $(64, 62, 4, 4, 3)$, then the apportionment at $H=136$ would have been $(63, 62, 4, 4, 3)$.  But this would imply,  $\{1, 5\} \subset U(\mathbf v, \mathbf h) \cap L(\mathbf v, \mathbf h)$, and it is easy to check that   $\frac{p_5}{3+s} > \frac{p_1}{63+s}$ for all $s \in [0, 0.5]$ and   $\frac{p_5}{\sqrt{12}} > \frac{p_1}{\sqrt{63 \times 64}}$. So seat $137$ would not have gone to state 1. This proves the claim.

Thus, the apportionment at  $H=137$ is either $(63, 63, 4, 4,3)$ or $(63, 62, 4, 4, 4)$. Calculation shows these both lead to an  apportionment at $H=140$ of $(64, 63, 5, 4, 4).$

To prove the second statement, note that at  $H = 172$, $\lfloor \mathbf q \rfloor =(79, 78, 5, 4,4)$, which sums to 170. By  Lemma \ref{limit2}, $h_3=5$ which means the only possible apportionments are $(80, 79, 5, 4, 4)$, $(80, 78, 5, 5, 4)$, $(79, 79, 5, 5, 4)$ and $(79, 78, 5, 5, 5)$.  Again, it is easy to check that each of these leads to an apportionment at $H=175$ of $(80,80,5,5,5).$
\end{proof}

\noindent {\bf Concluding remarks.}
The fact that quotatone methods fail proportional consistency indicates that, despite their  positive attributes, these methods have serious flaws. More generally, the behavior of quota methods with respect to  proportional consistency suggests that proportional consistency does not  capture the ``proportionality'' of an apportionment method but its ``consistency'' as $H$ is scaled,  since NIE, SUQ and LAR  satisfy proportional consistency but NIS, which stays closer to quota, does not.  This raises the question whether there is a better-formulated property that captures the idea  of an apportionment  method getting more proportional as the house size increases.

\end{document}